\documentclass[a4paper, 11pt]{article}

\usepackage{indentfirst}
\usepackage{amsmath,amsfonts,mathrsfs}

\newcommand{\ud}{\mathrm d} 

\newcommand{\dbar}{\overline{\partial}}

\newcommand{\ddbar}{\partial\overline{\partial}}

\newcommand{\R}{\mathbb R}
\newcommand{\C}{\mathbb C}
\newcommand{\N}{\mathbb N}

\newcommand{\CP}{\mathbb P}

\newcommand{\psh}[1]{\mathrm{PSH}(#1)} 

\newcommand{\supp}{\mathrm{Supp~}}

\newcommand{\hessian}{\mathrm{H}}
\newcommand{\smooth}{C^{\infty}} 

\newcommand{\kahler}{K\"ahler~}

\newcommand{\MA}{Monge-Amp\`ere~}

\usepackage{amsthm}
\usepackage{enumerate}
\usepackage{indentfirst}
\usepackage{amssymb}

\begin{document}

\theoremstyle{plain}
\newtheorem{theorem}{Theorem}[section]
\newtheorem{proposition}[theorem]{Proposition}
\newtheorem{lemma}[theorem]{Lemma}
\newtheorem{definition}[theorem]{Definition}
\newtheorem{fact}[theorem]{Fact}
\newtheorem{property}[theorem]{Property}
\newtheorem{corollary}[theorem]{Corollary}
\theoremstyle{definition}
\newtheorem*{remark}{Remark}
\newtheorem{example}[theorem]{Example}

\title{On the equivalence between local and global existence of complete \kahler metrics with plurisubharmonic potentials}
\author{Xu LIU}
\date{}

\maketitle

\begin{abstract}

\fontsize{10.5}{12}\selectfont

Like the classical potential theory, it was conjectured that there exists equivalence between locally and globally pluripolar and complete pluripolar sets, namely, Problem I of Lelong, and was solved by Josefson, Bedford - Taylor and Col\c{t}oiu. In this article, we consider complements of complete K\"ahler domains as the generalization of closed complete pluripolar sets and prove that there exists an equivalence between locally and globally existence of these sets.

\end{abstract}

\renewcommand{\thefootnote}{\fnsymbol{footnote}}
\footnotetext{\hspace*{-7mm} 
\begin{tabular}{@{}r@{}p{16.5cm}@{}}
& Date: \today. \\
& 2000 Mathematics Subject Classification. 32C25. 32U05.  \\
\end{tabular}}

\section{Introduction}

Holomorphic convexity / Steinness is an important topic in SCV and around it various theories have been developed. For example, E. E. Levi tried to find equivalent conditions of being domains of holomorphy via the description of their boundaries, i.e., Levi pseudoconvexity. This is the well-known Levi problem, which was solved in the affirmative first by K. Oka in 1940s. Following from that, many powerful and influential results were obtained in this direction. 

\bigskip

Another way to describe Stein manifolds is to use complete \kahler metrics, which was first considered by H. Grauert \cite{Grauert_1956}. 

A complete \kahler manifold $(M,\ud s^2)$ is a complex manifold $M$ together with a complete \kahler metric $\ud s^2$.
It is known that every Stein manifold carries a complete \kahler metric, i.e., it is a complete \kahler domain. A natural question is whether or not the converse holds, i.e., if $M$ is non-compact and complete K\"ahler, is $M$ necessarily Stein?  

It was observed by Grauert that the answer is negative. Instead, for any closed analytic subvariety $A$ of a Stein manifold $M$, he constructed a complete \kahler metric on $M \setminus A$ (Satz A in \cite{Grauert_1956}). 
In other words, in order to guarantee holomorphic convexity, besides existence of complete \kahler metrics, some additional assumptions are necessary , which generally divide into two kinds of approaches: curvature assumptions on the complete \kahler metrics; boundary regularity assumptions on the domains under consideration. 

In the same paper, Grauert showed that if $\Omega \subset M$ carries a complete \kahler metric and has a real-analytic boundary, then $\Omega$ is Stein (Satz C in \cite{Grauert_1956}). Later the regularity assumption was reduced from $C^{\omega}$ to $C^1$ by T. Ohsawa \cite{Ohsawa_1980_2}. 

\bigskip

In this article, we mainly follow the second line and study complete \kahler manifolds from the viewpoint of function theory. 
To be more precise, we study complete \kahler metrics by means of their potentials which are plurisubharmonic functions. In comparison with holomorphic functions, which are is some sense rigid, plurisubharmonic functions admit flexibility for modifications so that it is convenient to construct new functions as we desire.
At the same time, many notions in the classical potential theory can be generalized to several complex variables with subharmonic functions replaced by plurisubharmonic functions, e.g., pluripolar set, negligible sets, thin sets, etc. They are the important objects of study in the so-called pluripotential theory and play great roles in removable singularities and extension problems of analytic objects, etc.  

We will focus on these small sets. First it is not difficult to see the following propositions:  
\begin{enumerate}[(A)]
\item Except the trivial case (the whole domain), analytic sets are closed and complete pluripolar. 
\item Outside closed complete pluripolar sets one can construct complete \kahler \\
metrics. In other words, closed complete pluripolar sets are contained in the complements of complete \kahler domains. 
\end{enumerate}

Their relations can be shown in the following graph: 

\begin{eqnarray*}
&\{\text{analytic sets}\}& \\
&\cap& \\
&\{\text{closed complete pluripolar sets}\}& \subset~~~~~~~~ \{\text{pluripolar sets}\}\\
&\cap& ~~~~~~~~~~~~~~~~~~~~ \shortparallel \\
&\{\text{complements of complete K\"ahler domains}\}& ~~~~~~~~~~~ \{\text{negligible sets}\} 
\end{eqnarray*}

Now let us consider the converse of the inclusions in the left column. 

On complex-analyticity of real submanifolds as complements of complete \kahler domains, Ohsawa \cite{Ohsawa_1980_1} showed that for the two real codimensional case, merely $C^1$ regularity is sufficient. As a corollary, he also gave a partial answer to Nishino's problem \cite{Nishino_1962}, which can be seen as a partial converse to (A). It conjectured that if the graph of a continuous function is pluripolar, then the function is holomorphic. This problem was finally solved by N. Shcherbina \cite{Shcherbina_2005}. 

However, K. Diederich and J. E. Fornaess \cite{DF_1982_1} later considered the higher codimensional case  and showed that $C^{\omega}$-regularity is necessary. As counterexamples, they constructed a closed $\smooth$ submanifold $A$ of any real codimension $k \ge 3$ in a ball $B$, such that $A$ is not complex-analytic and $B \setminus A$ admits a complete \kahler metric. Later, we generalize their examples on open manifolds to the compact case. More precisely, for any $k \in \N, k \ge 3$, we construct a compact $\smooth$ submanifold $A$ of real codimension $k$ in $\CP^n$, such that $A$ is not complex-analytic and $\CP^n \setminus A$ admits a complete \kahler metric \cite{L_2015}. 

\bigskip

In light of the results in the classical potential theory, similar problems were posed for pluripotential theory, e.g., the equivalence between locally and globally pluripolar or complete pluripolar sets. However, different techniques were developed. 

In 1978, B. Josefson first showed the equivalence between local and global pluripolarity in Stein manifolds \cite{Josefson_1978}. Later E. Bedford and B. A. Taylor defined a new capacity with the help of complex \MA operators and gave an alternative proof \cite{BT_1976, BT_1982}. At the same time, they got the equivalence between pluripolarity and negligibility (the right column in the graph above). 
The similar problem for complete pluripolar sets was finally solved by M. Col\c{t}oiu in 1990, i.e., in Stein manifolds, a locally complete pluripolar set is also globally complete pluripolar \cite{Coltoiu_1989, Coltoiu_1990}. 

Inspired by the fact (B) and Col\c{t}oiu's result, we consider the following problem: 

\bigskip
\noindent{\bf Main Question.} \emph{Is it possible to patch up the potentials of complete \kahler metrics to obtain a global one?}
\bigskip

In other words, if a set is locally the complement of complete \kahler domains, is it globally the complement of some complete \kahler domain?  

\bigskip

For a precise setting, we start with the potentials instead of complete \kahler \\
metrics themselves. Otherwise, we need to extend the definition of the fundamental forms induced by these metrics so that we can solve the $\ddbar$-equations to obtain the potentials. The extension usually requires strong assumptions on the sets across which it is done. The known result is that the sets should be complete pluripolar. However, as the fact (B) mentioned above has shown, the existence of complete \kahler metrics outside complete pluripolar sets implies that the question has been solved in this case. So we choose a more general assumption and prove the following:   

\begin{theorem}\label{main_elckgck}
Assume $M$ is a Stein manifold and $A \subset M$ is a closed subset. 
If $M \setminus A$ locally admits complete \kahler metrics in the following sense: 
\begin{itemize}
\item $\{ U_i \}_{i \in \N}$ is a locally finite open covering of $M$; 
\item on each $U_i$, there exists $\varphi_i \in \psh{U_i} \cap \smooth(U_i \setminus A)$ such that $\ddbar \varphi_i$ gives a complete \kahler metric on $U_i \setminus A$ along $A \cap U_i$, 
\end{itemize}
then there exists a complete \kahler metric on $M \setminus A$ induced by a globally defined plurisubharmonic function on $M$. Moreover, this potential can be chosen to be bounded from below and smooth outside $A$. In particular, if every local potential is continuous, the global potential is also continuous. 
\end{theorem}

The idea of the proof is as follows: we divide the problem into two cases depending on whether the potentials are bounded. If all potentials are bounded, we use cut-off functions to extend their domains of definition to the whole of $M$. However, some negativities may be brought in this process. In order to remove them, we need to compose the strictly plurisubharmonic exhaustion function of $M$ with a suitably chosen increasing convex function. Then we can find a global potential. If on some open subset there exists a potential unbounded from below, we modify it to be bounded and reduce this case to the first one.     

\bigskip
\noindent{\bf Acknowledgements.} The author would like to express his deepest thanks to the supervisor Professor Takeo Ohsawa for academic guidance and discussions.
\bigskip

\section{Proof}

First we consider the case that every potential is bounded for below.

\begin{proposition}\label{bddp}
Assume $M$ is a Stein manifold and $A \subset M$ is a closed subset. 
If $M \setminus A$ locally admits complete \kahler metrics induced by bounded plurisubharmonic functions, i.e.,  
\begin{itemize}
\item $\{ U_i \}_{i \in \N}$ is a locally finite open covering of $M$; 
\item on each $U_i$, there exists $\varphi_i \in \psh{U_i} \cap \smooth(U_i \setminus A)$ such that $\varphi_i \ge 0$ and $\ddbar \varphi_i$ gives a complete \kahler metric on $U_i \setminus A$ along $A \cap U_i$, 
\end{itemize}
then there exists a complete \kahler metric on $M \setminus A$ induced by a globally defined plurisubharmonic function on $M$. Moreover, this potential can be constructed to be bounded from below. 
In particular, if every $\varphi_i$ is continuous, a continuous global potential can be chosen so that it is smooth outside $A$. 
\end{proposition}

\begin{proof}
It is known that plurisubharmonic function is always locally bounded from above. Since $\varphi_i$ is bounded, a linear transformation $t \mapsto at+b$ with $a > 0$ can be used to modify each $\varphi_i$ such that $1 \le \varphi_i \le 2$ for all $i$. Note that such modifications keep the completeness of the metrics. 

Set $u_i := \varphi_i^2$. It follows that 
$$
\ddbar u_i = 2(\partial \varphi_i \wedge \dbar \varphi_i + \varphi_i \ddbar \varphi_i). 
$$
So we know that $u_i \in \psh{U_i}$ also induces a complete \kahler metric on $U_i \setminus A$ along $A \cap U_i$. Moreover, we have the following estimate: 
$$\ddbar u_i \ge 2 \partial \varphi_i \wedge \dbar \varphi_i \ge \frac{1}{8} \partial \varphi_i^2 \wedge \dbar \varphi_i^2 = \frac{1}{8} \partial u_i \wedge \dbar u_i.$$ 

Note that on $A$, $\ddbar u_i$ should be understood in the sense of current. Since $u_i$ is bounded, we can choose a decreasing sequence of smooth plurisubharmonic functions $v_j$ which tends to $u_i$ and define $\ddbar u_i = \lim \ddbar v_j$. $\ddbar \varphi_i$ is defined in the same way. Then $\partial \varphi_i \wedge \dbar \varphi_i$ and $\partial u_i \wedge \dbar u_i$ are also defined and the same estimates hold.  

Take $U'_i \Subset U_i$ such that  $\{U'_i \}$ still forms an open covering of $M$. Choose $\rho_i \in \smooth(M)$ such that $\rho_i \ge 0, \supp \rho_i \subset U_i, \rho_i \equiv 1$ on $U'_i$. 
Then $\rho_i u_i$ extends to a function defined on $M$ and is smooth outside $A$. If $\rho_i \ne 0$, consider $\ddbar \rho_i u_i$ on $U_i \setminus A$: 
$$
\ddbar \rho_i u_i = u_i \ddbar \rho_i + \partial \rho_i \wedge \dbar u_i + \partial u_i \wedge \dbar \rho_i + \rho_i \ddbar u_i. 
$$

Since 
\begin{eqnarray*}
0 &\le& (\frac{4}{\sqrt{\rho_i}}\partial \rho_i + \frac{\sqrt{\rho_i}}{4} \partial u_i) \wedge (\frac{4}{\sqrt{\rho_i}}\dbar \rho_i + \frac{\sqrt{\rho_i}}{4} \dbar u_i) \\
&=& \frac{16}{\rho_i} \partial \rho_i \wedge \dbar \rho_i + \partial \rho_i \wedge \dbar u_i + \partial u_i \wedge \dbar \rho_i + \frac{\rho_i}{16} \partial u_i \wedge \dbar u_i, 
\end{eqnarray*}
it follows that 
\begin{eqnarray*}
\partial \rho_i \wedge \dbar u_i + \partial u_i \wedge \dbar \rho_i &\ge& -\frac{16}{\rho_i} \partial \rho_i \wedge \dbar \rho_i - \frac{\rho_i}{16} \partial u_i \wedge \dbar u_i \\
&\ge& -\frac{16}{\rho_i} \partial \rho_i \wedge \dbar \rho_i - \frac{\rho_i}{2} \ddbar u_i. 
\end{eqnarray*}
Therefore, 
$$
\ddbar \rho_i u_i \ge u_i \ddbar \rho_i -\frac{16}{\rho_i} \partial \rho_i \wedge \dbar \rho_i + \frac{\rho_i}{2} \ddbar u_i.  
$$
Note that $$\frac{16}{\rho_i} \partial \rho_i \wedge \dbar \rho_i \to 0 \text{~as~} \rho_i \to 0.$$ 

Since $M$ is Stein, there exists a strictly plurisubharmonic exhaustion function $\psi$ on $M$. Let 
$$
\varphi:= r(\psi(z)) + \sum_i \rho_i u_i(z), 
$$
where $r: \R \rightarrow \R$ is an increasing convex function. Locally there are only finite terms in the sum, so $\varphi$ is well defined. 

It is known that for each $c \in \R$, $\{\psi < c\} \Subset X$. Therefore, a large enough coefficient $C_c$ can be chosen such that $C_c\ddbar\psi$ removes the negativity brought by $\sum_i u_i \ddbar \rho_i -\frac{16}{\rho_i} \partial \rho_i \wedge \dbar \rho_i$ in $\{\psi < c\}$. If $r$ is chosen to increase rapidly enough at $+\infty$, then $\varphi$ is plurisubharmonic on $M$ such that $\hessian \varphi \ge \frac{1}{2}\hessian \varphi_i$ on $U'_i \setminus A$. Therefore, $\varphi$ induces a complete \kahler metric on $M \setminus A$. 

If every $\varphi_i$ is continuous, $\varphi$ constructed as above is also continuous. By setting $\mu$ a small positive constant and $\gamma = \hessian \varphi$, Richberg's regularization can be applied to $\varphi$ on $M \setminus A$ to obtain a smooth strictly plurisubharmonic function $\widetilde{\varphi}$. The estimate $\hessian \widetilde{\varphi} \ge (1-\mu)\hessian \varphi$ implies that $\widetilde{\varphi}$ induces a complete \kahler metric on $M \setminus A$. 
\end{proof}

\bigskip

For the case not every $\varphi_i$ is bounded from below, we need the following lemma to reduce it into the previous case. 
\begin{lemma}\label{unbddp} 
Under the same hypothesis as above, 
if on some open $U \subset M$, $\varphi \in \psh{U}\cap \smooth(U \setminus A)$ is unbounded and $\ddbar \varphi$ gives a complete \kahler metric on $U \setminus A$ along $A \cap U$, then there exists another potential $\tilde{\varphi} \in \psh{U}\cap \smooth(U \setminus A)$ such that $\tilde{\varphi}$ is bounded from below and $\ddbar \tilde{\varphi}$ gives a complete \kahler metric on $U \setminus A$ along $A$.
\end{lemma}

To verify completeness, the following Hopf-Rinow theorem is useful as a criterion. 

\begin{theorem}[Hopf-Rinow] 
Assume $(M,g)$ is a connected Riemannian manifold. The following are equivalent: 
\begin{enumerate}[(A)]
\item Any closed and bounded subset of $M$ is compact.
\item $M$ is complete as a metric space. 
\item $M$ is geodesically complete, i.e., for any $p$ in $M$, the exponential map $\exp_p$ is defined on the entire tangent space $T_p M$. 
\end{enumerate}
\end{theorem}
Condition (A) above is topologically equivalent to that any non relatively compact differential curve $\gamma$, i.e., $\gamma$ cannot be contained in any compact subset of $M$, has $\infty$ length with respect to $g$.

\begin{proof}
Consider $$\Phi_1 := e^{\varphi}, \Phi_2 := h(\varphi)$$ 
where $h(t) := \frac{1}{\log(-t)}\chi(t+3) + K\alpha(t)$, $\chi(t) \in \smooth(\R, [0,1])$ with $\chi \equiv 1$ on $(-\infty,0]$ and $\chi \equiv 0$ on $[1,\infty)$, $\alpha(t) \in \smooth(\R,[0,\infty))$ with $\alpha \equiv 0$ on $(-\infty,-4]$ and $\alpha''(t) > 0$ on $(-4,\infty)$, $K > 0$ is chosen large enough such that $h(t)$ is increasing and convex. It is clear that $\Phi_1, \Phi_2$ are plurisubharmonic on $U$ and nonnegative. 

Choose any differential curve $\gamma: [0,1) \to U \setminus A$ which is non relatively compact with respect to $M \setminus A$. 
If $\varphi\circ\gamma([0,1)) > C$ for some constant $C$, then the computation 
$$
\ddbar \Phi_1 = e^{\varphi}(\partial \varphi \dbar \varphi + \ddbar \varphi) \ge e^C \ddbar \varphi
$$
implies that the length of $\gamma$ with respect to $\ddbar \Phi_1 \ge e^{C/2} \cdot$ the length of $\gamma$ with respect to $\ddbar \varphi$. Since $\ddbar \varphi$ is a complete \kahler on $U \setminus A$ along $A$, which means the latter is $+\infty$, it follows that the former is also $+\infty$. 

If $\varphi\circ\gamma([0,1))$ is unbounded from below, 
when $\varphi\circ\gamma(t) < -1$, the following computation
\begin{eqnarray*}
& &\ddbar (\frac{1}{\log(-\varphi)})\\
 &=& \frac{2}{\log^3(-\varphi)}\frac{1}{\varphi^2}\partial \varphi \wedge \dbar \varphi + \frac{1}{\log^2(-\varphi)}\frac{1}{\varphi^2}\partial \varphi \wedge \dbar \varphi -\frac{1}{\log^2(-\varphi)}\frac{1}{\varphi}\ddbar \varphi\\
&\ge& \frac{1}{\log^2(-\varphi)}\frac{1}{\varphi^2}\partial \varphi \wedge \dbar \varphi
\end{eqnarray*}
implies that the length of $\gamma$ with respect to $\ddbar \Phi_2$, denoted by $l(\gamma)$, with some $t_0 \in (0,1)$ which is chosen such that $\varphi \circ \gamma(t_0) < -3$, has the estimate
\begin{eqnarray*}
l(\gamma) &\ge& \int_{t_0}^1 \sqrt{\sum \frac{1}{\log^2(-\varphi)}\frac{1}{\varphi^2}\frac{\partial \varphi}{\partial z_i}\frac{\ud z_i}{\ud t}\overline{\frac{\partial \varphi}{\partial z_j}\frac{\ud z_j}{\ud t}}} \ud t\\
&=& \int_{t_0}^1 \frac{1}{| \varphi\log(-\varphi) |} \huge| \frac{\ud \varphi\circ\gamma(t)}{\ud t}\huge| \ud t\\
&\ge& \liminf_{t \to 1} \int_{t_0}^t \frac{1}{-\varphi\log(-\varphi)} \ud -\varphi\circ\gamma(t)\\
&=& +\infty.
\end{eqnarray*}


Therefore, $\tilde{\varphi} := \Phi_1 + \Phi_2 \ge 0$ serves as the potential of a complete \kahler metric on $U \setminus A$ along $A$. 
\end{proof}

The main theorem follows immediately from Proposition \ref{bddp} and Lemma \ref{unbddp}. 

\section{Further questions}

Complete pluripolar sets serve as important examples of complements of complete \kahler domains. We want to consider what kind of set satisfies this condition and give more examples. 

And at the same time, we plan to consider what kind of set contains complements of complete \kahler domains as a subclass and whether or not the similar local and global equivalence holds there.

$$$$
Xu LIU\\ 
Graduate School of Mathematics, Nagoya University, Nagoya 464-8602, Japan\\
E-mail: x10001c@math.nagoya-u.ac.jp

\end{document}